\documentclass[11pt]{amsart}
\usepackage{amsfonts}
\usepackage{amsmath}
\usepackage{amssymb}
\usepackage{amsthm,amsfonts,latexsym,epsfig,geometry}
\usepackage{hyperref,times}
\usepackage{amscd}
\usepackage{mathtools}
\usepackage{color}

\allowdisplaybreaks[1]
\newtheorem{theorem}{Theorem}[section]
\newtheorem{lemma}[theorem]{Lemma}

\newtheorem{corollary}[theorem]{Corollary}
\theoremstyle{definition}
\newtheorem{definition}[theorem]{Definition}

\theoremstyle{remark}

\numberwithin{equation}{section}
\usepackage{hyperref}
\hypersetup{
	colorlinks=true,
	linkcolor={blue}, urlcolor={blue},
	citecolor={red}, anchorcolor = {blue}}
\begin{document}
\title[ Distribution of certain $\ell$-regular partitions  and  triangular numbers]{Distribution of certain $\ell$-regular partitions  and  triangular numbers}
%\title[$\ell$-regular partitions with distinct odd parts]{Divisibility of  $\ell$-regular partitions with distinct odd parts}
	\author{Chiranjit Ray}
	\address{Department of Mathematics, Harish-Chandra Research Institute, Prayagraj, Uttar Pradesh - 211 019, INDIA}
	\curraddr{}
	\email{chiranjitray.m@gmail.com}

	\date{July, 2021}
	\thanks{} 
	\begin{abstract} Let $pod_{\ell}(n)$ be the number of $\ell$-regular partitions of $n$  with distinct odd parts.  In this article, prove that for any positive integer $k$, the set of non-negative integers $n$ for which  $pod_{\ell}(n)\equiv 0 \pmod{p^{k}}$ has density one under certain conditions on $\ell$ and $p$. For $p \in \{3,5,7\}$, we also exhibit  multiplicative identities for $pod_{p}(n)$ modulo $p.$
	\end{abstract}
	\maketitle
\section{Introduction and statement of results}
	A partition of a positive integer $n$ is a non-increasing sequence of positive integers $\lambda_1, \lambda_2,\cdots, \lambda_k$ whose sum is $n$. Each $\lambda_i$ is called a part of the partition. If $\ell$ is a positive integer, then a partition is called an $\ell$-regular partition if there is no part divisible by $\ell$. Many mathematicians have studied this partition function and proved several interesting arithmetic and combinatorial properties; see \cite{Cui2013, Hou2015, Xia2015}. Furthermore, various other types of partition functions are studied in the literature by imposing certain restrictions on the parts of an $\ell$-regular partitions of $n$.
For example, Corteel and Lovejoy~\cite{Corteel2004} introduced
the notion of overpartition as a partition of $n$ in which the first occurrence of a number may be overlined. Ray et al.\ \cite{Ray2018, Ray2021} studied arithmetic and density properties of $\ell$-regular overpartitions. In \cite{Andrews1967,  Andrews1979, Berkovich2002, Hirschhorn2010}, authors considered the partitions with distinct odd parts, and even parts are unrestricted.

 If $n$ is a positive integer, then the $n$th triangular number is $T_n=\frac{n(n+1)}{2}$.
 In this article, we consider $\ell$-regular partitions of $n$ with distinct odd parts and even parts are unrestricted, and denote by  $pod_{\ell}(n)$.  The generating function of $pod_{\ell}(n)$ is as follows (see~\cite{Hemanthkumar2019}).
\begin{align}
		\label{gf1.1}	\sum_{n=0}^{\infty}pod_{\ell}(n)q^n&= \frac{\psi(-q^{\ell})}{\psi(-q)},
		%&=\frac{(q^{2}; q^{2})_{\infty}(q^{\ell}; q^{\ell})_{\infty}(q^{4\ell}; q^{4\ell})_{\infty}}{(q^{}; q^{})_{\infty}(q^{4}; q^{4})_{\infty}(q^{2\ell}; q^{2\ell})_{\infty}},
		%\label{gf1.2}\sum_{n=0}^{\infty}ped_{\ell}(n)q^n&=\frac{(q^{4}; q^{4})_{\infty}(q^{\ell}; q^{\ell})_{\infty}}{(q^{}; q^{})_{\infty}(q^{4\ell}; q^{4\ell})_{\infty}},
\end{align}
 where $\psi(q)= \sum_{n=0}^{\infty}q^{T_n}$. 
%Using binomial theorem we have $$\frac{\psi(-q^{p})}{\psi(-q)}\equiv \psi(-q)^{p-1} \pmod p$$, where $p$ is prime number. 
% Suppose $t_k$ denote the number of representations of $n$ as sum of $k$-triangular numbers. Then we have 
% $$\psi(-q)^k= \sum_{n=0}^{\infty}(-1)^nt_k(n)q^n$$
\\
%	where  the $q$-shifted factorial $(a; q)_{\infty}:=\displaystyle\prod_{n=1}^{\infty}(1-aq^{n-1}), ~~|q|<1.$

\par For a positive integer $M$ and $0\leq r \leq M$ , we define $$\displaystyle \delta_{r}^{g(n)}(M;X):= \frac{\#\{ 0\leq n \leq X: p(n)\equiv r \pmod M\}}{X},$$
where $g:\mathbb{N}\to \mathbb{Z}$ is an arbitrary function. Suppose $p(n)$  counts the number of integer partitions at $n.$
Parkin and Shanks \cite{PS} made the conjecture that the even and odd values of $p(n)$ are equally distributed, i.e, 
\begin{align*}
		\lim_{X\to\infty} \delta_{0}^{p(n)}(2;X) =\lim_{X\to\infty} \delta_{1}^{p(n)}(2;X)= \frac{1}{2}.
	\end{align*}	
Similar studies are done for some other partition functions, for example see \cite{Bringmann2008, Gordon1997, ray2020, Ray2021}. Recently,  Veena et  al.\ \cite{Veena2021} proved that the set of all positive integers such that $pod_3(n)\equiv 0 \pmod {3^k}$ have arithmetic density~$1$,~i.e.,
\begin{align}\label{{3}(n)}
		\lim_{X\to\infty} \delta_{0}^{{pod}_{3}(n)}(3^k;X) = 1.
\end{align}

Here we prove that for any odd positive integer $\ell$, the function ${pod}_{\ell}(n)$ is almost always divisible by $p^a$ with certain conditions, where $a$ is the largest positive integer for which $p^a$ divides $\ell$. In particular, we have the following result.
	\begin{theorem} \label{thm1}
	Let $\ell$ be an odd
integer greater than one, $p$ a prime divisor of $\ell$, and $a$ the largest
positive integer such that $p^a$
divides $\ell$. Then for every positive integer $k$ we have 
		\begin{align*}
		\lim_{X\to\infty} \delta_{0}^{{pod}_{\ell}(n)}(p^k;X) = 1.
		\end{align*}
	\end{theorem}
\noindent The above result is a generalization of \eqref{{3}(n)}. Furthermore, we get the following corollary as a direct consequence of the above theorem.
\begin{corollary}\label{cor1}
Let $p$ be an odd prime.
Then for any positive integer $k$, $pod_p(n)$ is almost always divisible by  $p^k$, i.e.,
\begin{align*}
	\lim_{X\to\infty} \delta_{0}^{{pod}_{p}(n)}(p^k;X) \equiv 1.	
\end{align*}
\end{corollary}

Next, we consider 3-regular partitions of $n$ with distinct odd parts $pod_3(n)$. In \cite{Gireesh}, Gireesh et  al.\  proved the following congruences:
	\begin{align*}
	    pod_3(n)&\equiv  pod_3(9n+2) \pmod{3^2},\\
	    pod_3(3n+2)&\equiv  pod_3(27n+20) \pmod{3^3},\\
	     pod_3(27n+20)&\equiv  pod_3(243n+182) \pmod{3^4}.
	\end{align*}
For any positive integers $k$ and $n$, Veena et  al.\ 	\cite{Veena2021} proved the following result 
% \begin{theorem}For any positive integers $k$ and $n$ the following congruence holds
 	\begin{align} \label{pod3.1}
 	{pod_3}\left(3^{2k}n+ \frac{3^{2k}-1}{4}\right) \equiv  ~~ {pod_3}(n) \pmod{3}.
 	\end{align}
% \end{theorem}
We have generalized \eqref{pod3.1} and deduced the following infinite families of multiplicative formulas for ${pod}_{3}(n)$ modulo $3$ using the theory of Hecke eigenforms. 
 \begin{theorem} \label{thm2.2_2}
 	Let $k$ be a positive integer, $p$ a prime with  $p \equiv 3 \pmod {4}$, and $\delta$  a non-negative integer
such that $p$ divides $4\delta+3$. Then for all $j\geq 0$ we have 
 	\begin{align*} 
 	{pod_3}\left(p^{k+1} n+ p\delta + \frac{3p-1}{4}\right) \equiv  {pod_3}\left(p^{k-1}n + \frac{4\delta +3-p}{4p}\right)  \pmod {3}.
 	\end{align*}
 \end{theorem}
 
 \begin{corollary}\label{cor2}
 	Let $k$ be a positive integer and $p$ a prime with  $p \equiv 3 \pmod {4}$. Then for all $n\geq 0$ we have 
 	\begin{align*} 	
 	{pod_3}\left(p^{2k}n+ \frac{p^{2k}-1}{4}\right) \equiv  ~~ {pod_3}(n) \pmod{3}.
 	\end{align*}
 \end{corollary}
  \begin{corollary}\label{cor3}
 		Let $k$ be a positive integer and $p$ a prime with  $p \equiv 3 \pmod {4}$. Then for all $n\geq 0$ we have  
 	\begin{align*} 	
 	{pod_3}\left(p^{2k+1}n+\frac{p^{2k}-1}{4}\right) \equiv  ~~ {pod_3}(np) \pmod{3}.
 	\end{align*}
 \end{corollary}
  \noindent For example, if we consider $p=7$, and $k=1$,  then from Corollary~\ref{cor3} we have the following congruences:
 \begin{align*} 
%  {pod_3}\left(49n+12\right) \equiv{pod_3}\left(9n+2\right) &\equiv~{pod_3}\left(n\right),~ \text{and}\\
%   {pod_3}\left(27n+2\right) &\equiv{pod_3}\left(3n\right),\\
 {pod_3}\left(343n+24\right) &\equiv{pod_3}\left(7n\right)\pmod 3.
  \end{align*}
%  Some of the above congruences are also verified by Gireesh et al.\  \cite{Gireesh}.

In addition, we prove the following results for  
$pod_5$ and $pod_7$. Let  $\sigma_k(n)=\sum_{d|m}d^k$ be the standard divisor function and $\sigma_{2,\chi}(n)=\sum_{d|m}\chi(d)d^2$ be the generalized divisor function, where $\chi$ is the Dirichlet character modulo $4$  with $\chi(1)=1$ and $\chi(3)=-1$.
\begin{theorem}\label{pod5} For any positive integer $n$, we have
\begin{align}
 \label{con_pod_5} pod_5(n)&\equiv(-1)^n \sigma_1(2n+1) \pmod 5, ~\text{and}\\
\label{con_pod_7}     pod_7(n)&\equiv(-1)^{n+1} \frac{1}{8} \sigma_{2,\chi}(4n+3) \pmod 7.
\end{align}
\end{theorem}
We use  $Mathematica$ \cite{mathematica} for our computations.
%%%%%%%%%%%%%%%%%%%%%%%%%%%%%%%%%%%%%%%%%%%%%%%%%%%%%%%%%%%%%%%%%%%%%%%%55
	\section{Preliminaries}
	In this section, we recall some necessary facts and notation coming from modular forms (for further details we refer the reader to \cite{koblitz1993} and \cite{ono2004}).  For a positive integer $k$ denote by  $M_{k}(\Gamma)$ the complex vector space of modular forms of weight $k$ for a congruence subgroup $\Gamma$,  and let $\mathbb{H}$ be the complex upper half plane.
	\begin{definition}\cite[Definition 1.15]{ono2004}
		Let $\chi$ be a Dirichlet character modulo $N$. Then a modular form $f\in M_{\ell}(\Gamma_1(N))$ has Nebentypus character $\chi$ if
		$$f\left( \frac{az+b}{cz+d}\right)=\chi(d)(cz+d)^{\ell}f(z)$$ for all $z\in \mathbb{H}$ and all $\begin{bmatrix}
			a  &  b \\
			c  &  d      
		\end{bmatrix} \in \Gamma_0(N)$. We denote the space of such modular forms by $M_{\ell}(\Gamma_0(N), \chi)$. 
	\end{definition}
	\noindent Recall that Dedekind's eta-function is defined by
		$\eta(z):=q^{1/24}(q;q)_{\infty},$
where $q=e^{2\pi iz}$ and $z\in \mathbb{H}$. 
A function $f(z)$ is called an $eta$-quotient if it is expressible as a finite product of
the form $$f(z)=\prod_{\delta\mid N}\eta(\delta z)^{r_\delta},$$ where $N$ is a positive integer and each $r_{\delta}$ is an integer. The following two theorems allow one to determine whether a given eta-quotient is a modular form.
	\begin{theorem}\cite[Theorem 1.64]{ono2004}\label{thm_ono1} Suppose that $f(z)=\displaystyle\prod_{\delta\mid N}\eta(\delta z)^{r_\delta}$ 
		is an eta-quotient such that 
		\begin{eqnarray*}	
			\ell&=&\displaystyle\frac{1}{2}\sum_{\delta\mid N}r_{\delta}\in \mathbb{Z},\\
			\sum_{\delta\mid N} \delta r_{\delta}&\equiv& 0 \pmod{24} ~~\mbox{and}\\
			\sum_{\delta\mid N} \frac{N}{\delta}r_{\delta}&\equiv& 0 \pmod{24}.
		\end{eqnarray*}
		Then 
		$$
		f\left( \frac{az+b}{cz+d}\right)=\chi(d)(cz+d)^{\ell}f(z)
		$$
		for every  $\begin{bmatrix}
			a  &  b \\
			c  &  d      
		\end{bmatrix} \in \Gamma_0(N)$. Here 
		%$\chi$ is defined by 
		$$\chi(d):=\left(\frac{(-1)^{\ell} \prod_{\delta\mid N}\delta^{r_{\delta}}}{d}\right).$$
	\end{theorem}
\noindent If the eta-quotient $f(z)$ satisfy the conditions of Theorem \ref{thm_ono1} and  holomorphic at all of the cusps of $\Gamma_0(N)$, then $f\in M_{\ell}(\Gamma_0(N), \chi)$. To determine  the orders of an eta-quotient at each cusp is the following.
	\begin{theorem}\cite[Theorem 1.65]{ono2004}\label{thm_ono1.1}
		Let $c, d,$ and $N$ be positive integers with $d\mid N$ and $\gcd(c, d)=1$. If $f(z)$ is an eta-quotient satisfying the conditions of Theorem~\ref{thm_ono1} for $N$, then the order of vanishing of $f(z)$ at the cusp $\frac{c}{d}$ 
		is $$\frac{N}{24}\sum_{\delta\mid N}\frac{\gcd(d,\delta)^2r_{\delta}}{\gcd(d,\frac{N}{d})d\delta}.$$
	\end{theorem}
\noindent Let $m$ be a positive integer and $f(z) = \displaystyle \sum_{n=0}^{\infty} a(n)q^n \in M_{\ell}(\Gamma_0(N),\chi)$. Then the action of Hecke operator $T_m$ on $f(z)$ is defined by 
	\begin{align}
\label{hecke1}	f(z)|T_m := \sum_{n=0}^{\infty} \left(\sum_{d\mid \gcd(n,m)}\chi(d)d^{\ell-1}a\left(\frac{nm}{d^2}\right)\right)q^n.
	\end{align}
We note that $a(n)=0$ unless $n$ is a non-negative integer. The modular form $f(z)$ is called a Hecke eigenform if for every $m\geq2$ there exists a complex number $\lambda(m)$ for which 
	\begin{align}\label{hecke3}
	f(z)|T_m = \lambda(m)f(z).
	\end{align}
%%%%%%%%%%%%%%%%%%%%%%%%%%%%%%%%%%%%%%%%%%
\section{Proofs of Theorem \ref{thm1}}
 To prove Theorem \ref{thm1}, we need the following lemmas.
	\begin{lemma}\label{lem2} Let $\ell$ be an odd integer greater than one, $p$ a prime divisor of $\ell$, and $a$ the largest positive integer such that $p^a$
divides $\ell$.  Then for any positive integer $k$ we have 
		\begin{align*}
			\frac{\eta(24z)^{p^{a+k}-1}\eta(48z)\eta(24\ell z)\eta(96\ell z)}{\eta(96z)\eta(48\ell z)\eta(24p^{a}z)^{p^{k}}} \equiv \sum_{n=0}^{\infty}pod_{\ell}(n)q^{24n+3(\ell-1)} \pmod {p^k}.
		\end{align*}
	\end{lemma}
	\begin{proof} 	
		Consider 
		\begin{align*}
			\mathcal{A}(z) 
			=\prod_{n=1}^{\infty}\frac{(1-q^{24n})^{p^{a}}}{(1-q^{24p^{a}n})}  =\frac{\eta(24z)^{p^{a}}}{\eta(24p^{a}z)}.
		\end{align*}
		By the binomial theorem, for any positive integers $r$, $k$, and prime $p$ we have
		\begin{align*}
			(q^{r};q^{r})_{\infty}^{p^k}\equiv (q^{pr};q^{pr})_{\infty}^{p^{k-1}} \pmod {p^{k}}.
		\end{align*}
		Therefore,
		\begin{align*}
			\mathcal{A}^{p^{k}}(z) = \frac{\eta(24z)^{p^{a+k}}}{\eta(24p^{a}z)^{p^k}} \equiv 1 \pmod {p^{k+1}}.
		\end{align*}
		Define $\mathcal{B}_{\ell,p, k}(z)$ by
		
		$$\mathcal{B}_{\ell,p, k}(z)=	\frac{\eta(48z)\eta(24\ell z)\eta(96\ell z)}{\eta(24z)\eta(96z)\eta(48\ell z)}~\mathcal{A}^{p^{k}}(z).$$
		
		\noindent Now, modulo $p^{k+1}$, we have 
		\begin{align}\label{new-110}
			\notag	\mathcal{B}_{\ell,p, k}(z)	&=\frac{\eta(48z)\eta(24\ell z)\eta(96\ell z)}{\eta(24z)\eta(96z)\eta(48\ell z)} \frac{\eta(24z)^{p^{a+k}}}{\eta(24p^{a}z)^{p^k}}\\
			\notag	&\equiv \frac{\eta(48z)\eta(24\ell z)\eta(96\ell z)}{\eta(24z)\eta(96z)\eta(48\ell z)}\\
		%	\notag&= q^{3(\ell-1)}~\frac{(q^{48}; q^{48})_{\infty}(q^{24\ell}; q^{24\ell})_{\infty}(q^{96\ell}; q^{96\ell})_{\infty}}{(q^{24}; q^{24})_{\infty}(q^{96}; q^{96})_{\infty}(q^{48\ell}; q^{48\ell})_{\infty}}\\
			&= q^{3(\ell-1)}~\frac{\psi(-q^{24\ell})}{\psi(-q^{24})}
		\end{align}
		Since $$\mathcal{B}_{\ell,p, k}(z)=\frac{\eta(24z)^{p^{a+k}-1}\eta(48z)\eta(24\ell z)\eta(96\ell z)}{\eta(96z)\eta(48\ell z)\eta(24p^{a}z)^{p^{k}}},$$ combining \eqref{gf1.1} and \eqref{new-110}, we obtain the required result.
	\end{proof}
	\begin{lemma}\label{lem1}
		Let $\ell$ be an odd
integer greater than one, $p$ a prime divisor of $\ell$, and $a$ the largest positive integer such that $p^a$ divides $\ell$.  Then, for a positive integer $k>a$, we have
		\begin{align*}
			\mathcal{B}_{\ell,p, k}(z) \in M_{\frac{p^k(p^{a}-1)}{2}}\left(\Gamma_0(384 \cdot \ell), \chi(\bullet)\right),
		\end{align*}
	where the Nebentypus character $$\chi(\bullet)=\left(\frac{(-1)^{\frac{p^k(p^{a}-1)}{2}} (24)^{p^{a+k}-1}\cdot 48 \cdot 24\ell \cdot 96\ell \cdot (96)^{-1} \cdot (48\ell)^{-1} \cdot (24p^{a})^{p^k}}{\bullet}\right).$$	
	\end{lemma}
	\begin{proof}  First we verify the first, second and third hypotheses of Theorem \ref{thm_ono1}. The weight of the $eta$-quotient $\mathcal{B}_{\ell,p, k}(z)$ is $\frac{1}{2}(p^{a+k}-p^{k})=\frac{p^k}{2}(p^{a}-1).$\\ Suppose the level of the $eta$-quotient $\mathcal{B}_{\ell,p, k}(z)$ is $96\ell u$, where $u$ is the smallest positive integer satisfying the following identity.
$$(p^{a+k}-1)~\frac{96\ell u}{24}+\frac{96\ell u}{48}+\frac{96\ell u}{24\ell}+\frac{96\ell u}{96\ell}-\frac{96\ell u}{96}-\frac{96\ell u}{48\ell}-p^{k}~\frac{96\ell u}{24p^{a}}\equiv0\pmod{24}.$$
Equivalently, we have
\begin{equation}
\label{ul}
    u\left(4\ell p^{k-a}\left(p^{2a}-1\right)-3(\ell-1)\right)\equiv0\pmod{24}.
\end{equation}
For all prime $p\neq 3$, note that $p^{2a}-1$ is multiple of $3$. Hence, from \eqref{ul}, we conclude that the level of the $eta$-quotient $\mathcal{B}_{\ell,p, k}(z)$ is $384 \ell$ for $k>a$.

% \item  The Nebentypus character is  $$\chi(\bullet)=\left(\frac{(-1)^{\frac{p^k(p^{a}-1)}{2}} (24)^{p^{a+k}-1}\cdot 48 \cdot 24\ell \cdot 96\ell \cdot (96)^{-1} \cdot (48\ell)^{-1} \cdot (24p^{a})^{p^k}}{\bullet}\right).$$
	
		By Theorem \ref{thm_ono1.1}, the cusps of $\Gamma_{0}(384\ell)$ are given by  $\frac{c}{d}$ where $d~\mid~384\ell$ and $\gcd(c, d)~=~1$. Now  $\mathcal{B}_{\ell,p, k}(z)$ is holomorphic at a cusp $\frac{c}{d}$ if and only if
		\begin{align*}
			(p^{a+k}-1)~&\frac{\gcd(d,24)^2}{24}+ \frac{\gcd(d,48)^2}{48}+\frac{\gcd(d,24\ell)^2}{24\ell}+\frac{\gcd(d,96\ell)^2}{96\ell}\\
			&-\frac{\gcd(d,96)^2}{96}-\frac{\gcd(d,48\ell)^2}{48\ell}-p^{k}~\frac{\gcd(d,24p^{a})^2}{24p^{a}}\geq 0.
		\end{align*}
		Equivalently, if and only if 
		\begin{align}\label{c1}	
			\notag	&4\ell(p^{a+k}-1)~\frac{\gcd(d,24)^2}{\gcd(d,96\ell)^2}+2\ell~\frac{\gcd(d,48)^2}{\gcd(d,96\ell)^2}+4~\frac{\gcd(d,24\ell)^2}{\gcd(d,96\ell)^2}\\
			&-\ell~\frac{\gcd(d,96)^2}{\gcd(d,96\ell)^2}-2~\frac{\gcd(d,48\ell)^2}{\gcd(d,96\ell)^2}-4\ell p^{k-a}\frac{\gcd(d,24p^{a})^2}{\gcd(d,96\ell)^2}+1\geq 0.
		\end{align}
		
	\noindent	To check the positivity of \eqref{c1}, we have 
		to find all the possible divisors of $384\ell$. We define three sets as follows. 
		\begin{align*}
			\mathcal{H}_1&=\{2^{r_1}3^{r_2}tp^{s}: 0\leq r_1 \leq3, 0\leq r_2\leq1, t|\ell~~\text{but}~~ p~\nmid~t,~~and~~ 0\leq s \leq a\}, \\
			\mathcal{H}_2&=\{2^{r_1}3^{r_2}tp^{s}: r_1=4, 0\leq r_2\leq1, t|\ell~~\text{but}~~ p~\nmid~t,~~and~~ 0\leq s \leq a\}, \\
			\mathcal{H}_3&=\{2^{r_1}3^{r_2}tp^{s}: 5\leq r_1 \leq7, 0\leq r_2\leq1, t|\ell~~\text{but}~~ p~\nmid~t,~~and~~ 0\leq s \leq a\}.
		\end{align*}
		Note that $\mathcal{H}_1\cup \mathcal{H}_2\cup \mathcal{H}_3$ contains all positive divisors of $384\ell$. In the following table, we compute all necessary data to prove the positivity of \eqref{c1}.
		\begin{center}
			
			\begin{tabular}{|p{1.5cm}|p{2cm}|p{2cm}|p{2cm}|p{2.3cm}|p{2cm}|p{2cm}|}
				\hline
				\vspace{0.0001 in}Values of $d$ such that \vspace{0.0001 in} $d|384\ell$ &\vspace{0.0001 in} $\dfrac{\gcd(d,24)^2}{\gcd(d,96\ell)^2}$ \vspace{0.0001 in}&\vspace{0.0001 in}$\dfrac{\gcd(d,48)^2}{\gcd(d,96\ell)^2}$ &\vspace{0.0001 in}$\dfrac{\gcd(d,96)^2}{\gcd(d,96\ell)^2}$ & \vspace{0.0001 in}$\dfrac{\gcd(d,24p^{a})^2}{\gcd(d,96\ell)^2}$ & \vspace{0.0001 in}$\dfrac{\gcd(d,24 \ell)^2}{\gcd(d,96\ell)^2}$& \vspace{0.0001 in}$\dfrac{\gcd(d,48\ell)^2}{\gcd(d,96\ell)^2}$\\	
				\hline	
				\vspace{0.001 in} $d\in\mathcal{H}_1$ &~~ \vspace{0.0001 in} $1/t^2p^{2s}$   &~~ \vspace{0.0001 in} $1/t^2p^{2s}$&~~ \vspace{0.0001 in}$1/t^2p^{2s}$&~~ \vspace{0.0001 in}$1/t^2$&~~ \vspace{0.0001 in}$1$&~~ \vspace{0.0001 in}$1$\\
				\hline
				\vspace{0.001 in} $d\in\mathcal{H}_2$&~~\vspace{0.0001 in} $1/4t^2p^{2s}$  &~~ \vspace{0.0001 in} $1/t^2p^{2s}$ &~~ \vspace{0.0001 in} $1/t^2p^{2s}$& ~~ \vspace{0.0001 in} $1/4t^2$&~~ \vspace{0.0001 in} $1/4$&~~ \vspace{0.0001 in} $1$\\
				\hline
				\vspace{0.001 in} $d\in\mathcal{H}_2$&~~\vspace{0.0001 in} $1/16t^2p^{2s}$  &~~ \vspace{0.0001 in} $1/4t^2p^{2s}$ &~~ \vspace{0.0001 in} $1/t^2p^{2s}$& ~~ \vspace{0.0001 in} $1/16t^2$&~~ \vspace{0.0001 in} $1/16$&~~ \vspace{0.0001 in} $1/4$\\
				\hline
			\end{tabular}	
		
		\end{center}
		
		\noindent \textbf{Case (i).} If $d\in\mathcal{H}_1$, then left side of \eqref{c1} can be written as: 
		\begin{align}
			\label{1.n.1}\frac{\ell}{t^2}\left[4p^{k}\left(\frac{p^{a}}{p^{2s}}-\frac{1}{p^{a}}\right)-3\frac{1}{p^{2s}}\right]+3.
		\end{align}
		When $s=a$, the above quantities, $\left(3-\frac{3\ell}{t^2p^{2a}}\right)\geq 0$, as $p^{2a}\geq \ell$. For $0 \leq s < a $ it is clear that $\frac{p^{a}}{p^{2s}}-\frac{1}{p^{a}}>0$. Therefore,
		\begin{align*}
			\frac{p^{a}}{p^{2s}}-\frac{1}{p^{a}}-\frac{1}{p^{2s}}\geq \frac{ p^{2a}-p^{2(a-1)}-p^{a}}{p^{a+2s}}= \frac{p^{a}\left[p^{a}(1-\frac{1}{p^2})-1\right]}{p^{a+2s}}>0.
		\end{align*}
		The last inequality holds because $p^{a}(1-\frac{1}{p^2})>1$ for all prime $p$. Hence the quantities  in \eqref{1.n.1} are greater than or equal to $0$ when $p^{2a}\geq \ell$.\\
		
		\noindent \textbf{Case (ii).} If $d\in\mathcal{H}_2$ or $d\in\mathcal{H}_3$, then left side of \eqref{c1} can be written respectively as:
		
		\begin{align}
			\label{1.n.3}  &\frac{p^{k}\ell}{t^2}\left(\frac{p^{a}}{p^{2s}}-\frac{1}{p^{a}}\right), ~~\text{and}~~\\	
			\label{1.n.2}	&\frac{\ell}{4t^2}\left[p^{k}\left(\frac{p^{a}}{p^{2s}}-\frac{1}{p^{a}}\right)-3\frac{1}{p^{2s}}\right]+\frac{3}{4}.
		\end{align}
		By the similar argument as case~ (i), the quantities  in \eqref{1.n.2} and \eqref{1.n.3} are greater than or equal to~$0$.\\
		
	\noindent Therefore, by \textbf{Case (i)} and \textbf{Case (ii)}, the orders of vanishing of $\mathcal{B}_{\ell,p, k}(z)$ at the at the cusp $\frac{c}{d}$ is nonnegative. So  $\mathcal{B}_{\ell,p, k}(z)$ is holomorphic at every cusp $\frac{c}{d}$. We have also verified the Nebentypus character  by Theorem \ref{thm_ono1}. Hence $\mathcal{B}_{\ell,p, k}(z)$ is a modular form of weight $\frac{p^k(p^{a}-1)}{2}$ on $\Gamma_0(384\cdot \ell)$ with  Nebentypus character $\chi(\bullet)$. 
	\end{proof}
	We state the following result of Serre, which is useful to prove  Theorem \ref{thm1}.
	\begin{theorem}\cite[Theorem~2.65]{ono2004}\label{serre}
		Let $k, m$  be positive integers. If  $f(z)\in M_{k}(\Gamma_0(N), \chi(\bullet))$ has the Fourier expansion $f(z)=\sum_{n=0}^{\infty}c(n)q^n\in \mathbb{Z}[[q]],$
		then there is a constant $\alpha>0$  such that
		$$
		\# \left\{n\leq X: c(n)\not\equiv 0 \pmod{m} \right\}= \mathcal{O}\left(\frac{X}{\log^{\alpha}{}X}\right).
		$$
	\end{theorem} 
	\begin{proof}[Proof of Theorem \ref{thm1}]Suppose $k>1$ is a positive integer. From Lemma~\ref{lem1}, we have 
		$$	\mathcal{B}_{\ell,p, k}(z)=\frac{\eta(24z)^{p^{a+k}-1}\eta(48z)\eta(24\ell z)\eta(96\ell z)}{\eta(96z)\eta(48\ell z)\eta(24p^{a}z)^{p^{k}}} \in M_{\frac{p^k(p^{a}-1)}{2}}\left(\Gamma_0(384\cdot \ell), \chi(\bullet)\right).$$  Also the Fourier coefficients of the $eta$-quotient $\mathcal{B}_{\ell,p, k}(z)$ are integers. So, by Theorem \ref{serre} and Lemma \ref{lem2}, we can find a constant $\alpha>0$ such that
		$$
		\# \left\{n\leq X:pod_{\ell}(n)\not\equiv 0 \pmod{p^k} \right\}= \mathcal{O}\left(\frac{X}{\log^{\alpha}{}X}\right).
		$$
		Hence  $$\lim\limits_{X\to +\infty}\frac{	\# \left\{n\leq X:pod_{\ell}(n)\equiv 0 \pmod{p^k} \right\}}{X}=1.$$
This allows us to prove the required divisibility by $p^k$ for all $k>a$ which trivially gives divisibility by $p^k$ for all positive integer $k\leq a$. This completes the proof of Theorem \ref{thm1}.
	\end{proof}
	
%%%%%%%%%%%%%%%%%%%%%%%%%%%%%%%%%%%%%%%%%%%%%%%%%%%%%%%%%%%%%%%%%%%	
\section{Proofs of Theorem \ref{thm2.2_2} - \ref{pod5}}
\noindent In this  section,  we discuss about the multiplicative nature of  $pod_{3}(n)$ and we get Corollary~\ref{cor2} - \ref{cor3} as a special case of Theorem~\ref{thm2.2_2}.
 Then we obtain some congruence formulas for $pod_{5}(n)$ and $pod_{7}(n)$ in the Theorem \ref{pod5}-\ref{pod5}.

\begin{proof}[Proof of Theorem~\ref{thm2.2_2}].By \eqref{gf1.1}  we have 
\begin{align}
\label{gf1.1.1} \sum_{n=0}^{\infty}pod_3(n)q^n &= \frac{\psi(-q^3)}{\psi(-q)}\equiv \psi(-q)^2 \pmod{3}.
\end{align}
  Using Theorem \ref{thm_ono1} and Theorem \ref{thm_ono1.1} we have  $\dfrac{\eta(4z)^2 \eta(16z)^2}{\eta(8z)^2}\in M_1(\Gamma_0(64),\chi_{-1}(\bullet))$, where $\chi_{-1}$ is defined by $\chi_{-1}(\bullet)=(\frac{-1}{\bullet}).$ Therefore the above eta-quotient has a Fourier expansion and consider
\begin{align}\label{gf1.1.2}
   \frac{\eta(4z)^2 \eta(16z)^2}{\eta(8z)^2} = q - 2 q^5 + 3 q^9 - 6 q^{13} + 7 q^{17} - 10 q^{21}+ \cdots =\displaystyle\sum_{n=1}^{\infty} a(n)q^n. 
\end{align}
From \eqref{gf1.1.1} and \eqref{gf1.1.2}, for any positive integer $n$, we have
\begin{align}\label{gf1.1.3}
pod_{3}(n)\equiv  a(4n+1) \pmod{3}.
\end{align}
 From \cite{martin},  we know that  $\frac{\eta(4z)^2 \eta(16z)^2}{\eta(8z)^2}$ is a Hecke eigenform.
Using \eqref{hecke1} and \eqref{hecke3} we obtain
\begin{align*}
	\left.\frac{\eta(4z)^2 \eta(16z)^2}{\eta(8z)^2}\right\vert{T_p} = \sum_{n=1}^{\infty} \left(a(pn) + \left(\frac{-1}{p}\right) a\left(\frac{n}{p}\right) \right)q^n = \lambda(p) \sum_{n=1}^{\infty} a(n)q^n.
\end{align*}
Since $a(p)=0$ for all $p\equiv3\pmod{4}$, equating the coefficients on the both sides, we have $\lambda(p)=0$, and
	\begin{align}\label{thm2.2.2}
	a(pn) + \left(\frac{-1}{p}\right) a\left(\frac{n}{p}\right)= 0.
	\end{align}
Replacing $n$ by $4p^kn+4\delta+3$ in \eqref{thm2.2.2}, we get the following 
\begin{align}\label{thm2.2.4}
a\left(4\left( p^{k+1}n+ p\delta +\frac{3p-1}{4}\right)+1\right)= (-1)\left(\frac{-1}{p}\right) a\left(4\left(p^{k-1}n+ \frac{4\delta +3-p}{4p}\right)+1\right).
\end{align}	
Note that  $\frac{3p-1}{4}$ and $\frac{4\delta+3-p}{4p}$  are integers. Using \eqref{gf1.1.3} and \eqref{thm2.2.4} we obtain
	
	\begin{align} \label{thm2.2.4.11}
	pod_{3}\left(p^{k+1} n+ p\delta + \frac{3p-1}{4}\right) \equiv(-1)\left(\frac{-1}{p}\right)pod_{3}\left(p^{k-1}n + \frac{4\delta +3-p}{4p}\right)  \pmod {3}.
	\end{align}
Since $\left(\frac{-1}{p}\right)=-1$, for all prime $p\equiv 3 \pmod{4}$, our result now follows from~\eqref{thm2.2.4.11}.
\end{proof}	
\begin{proof}[Proof of Corollary~\ref{cor2}]
Let $p\equiv 3\pmod {4}$ be a prime. 
Now replacing $n$ with $np^{k-1}$ in \eqref{thm2.2.4.11}
and then considering $4\delta+3=p^{2k-1}$ we have
	\begin{align*}
	pod_{3}\left(\frac{p^{2k}(4n+1)-1}{4}\right) \equiv  pod_{3}\left(\frac{p^{2(k-1)}(4n+1)-1}{4}\right) \pmod 3.
	\end{align*}
	By repeating the above relation for ($k - 1$)-times, we obtain 	
	\begin{align*}
	pod_{3}\left(\frac{p^{2k}(4n+1)-1}{4}\right) \equiv  pod_{3}(n) \pmod 3.
	\end{align*}		
	Then the result directly follows from the above congruence.
\end{proof}	
\begin{proof}[Proof of Corollary~\ref{cor3}] We can prove Corollary~\ref{cor3}  in a similar fashion as Corollary~\ref{cor2}.
\end{proof}	
 For a positive integer $k$, let $t_k(n)$ denote the number of representations of $n$ as a sum of $k$ triangular numbers.
 It is easy to see that if $k > 1$ then
 $$\psi(q)^k= \sum_{n=0}^{\infty}t_k(n)q^n.$$
 If $p$ an odd prime number, then applying the binomial theorem on \eqref{gf1.1} we have $$\sum_{n=0}^{\infty}pod_{p}(n)q^n\equiv \psi(-q)^{p-1} \pmod p.$$
 Therefore for a positive integer $n$ and a odd prime $p$ we obtain 
 \begin{align} \label{podp}
 pod_{p}(n)\equiv(-1)^nt_{p-1}(n) \pmod p.
 \end{align}
 \begin{proof}[Proof of Theorem~\ref{pod5}] Suppose $t_4(n)$ is the number of representations of $n$ as a sum of $4$ triangular numbers. Then Ono et al.\ \cite[Theorem 3]{ono1995} showed that  
 $$t_4(n)=\sigma_1(2n+1).$$
 Now \eqref{con_pod_5} directly follows from the above equation and \eqref{podp}. 
Similarly, the result \eqref{con_pod_7} follows from  \cite[Theorem 4]{ono1995} and \eqref{podp}.
 \end{proof}
\subsection*{Acknowledgements} 
The author has carried out this work at Harish-Chandra Research Institute (India) as a Postdoctoral Fellow. 
\bibliographystyle{plain}
\bibliography{regular}
\end{document}